\documentclass[11pt]{amsart}

\input{xy}
\xyoption{all}
\usepackage{graphicx}
\usepackage{color}
\usepackage{verbatim}
\usepackage{hyperref}

\newtheorem{theorem}{Theorem} [section]
\newtheorem{prop}[theorem]{Proposition}
\newtheorem{lemma}[theorem]{Lemma}
\newtheorem{cor}[theorem]{Corollary}

\theoremstyle{definition}

\theoremstyle{remark}
\newtheorem{remark}{Remark}

\numberwithin{equation}{section}
\numberwithin{figure}{section}
\numberwithin{example}{section}

\newcommand\A{{\mathbb A}}
\newcommand\bP{{\mathbb P}}
\newcommand\C{{\mathbb C}}

\renewcommand\P{{\mathbb P}}

\newcommand\Q{{\mathbb Q}}

\newcommand\D{{\mathbb D}}

\renewcommand\phi{\varphi}

\renewcommand\O{\mathcal{O}}
\newcommand\cL{\mathcal{L}}
\newcommand\cLbar {\mathcal{\overline{L}}}

\newcommand\cM{\mathcal{M}}

\newcommand\Gal{\operatorname{Gal}}

\newcommand{\bff}{{\mathbf f}}

\renewcommand\mod{\operatorname{mod}}  

\newcommand\OK {\Omega_K}

\newcommand\lra{\longrightarrow}

\newcommand\Kbar {\overline{K}}
\newcommand\Qbar {\overline{\Q}}

\newcommand\hhat {\hat{h}}

\begin{document}

\title{The Dynamical Andr\'e-Oort Conjecture: Unicritical Polynomials}

\author{D. Ghioca}
\address{
Dragos Ghioca\\
Department of Mathematics\\
University of British Columbia\\
Vancouver, BC V6T 1Z2\\
Canada
}
\email{dghioca@math.ubc.ca}

\author{H. Krieger}
\address{
Holly Krieger\\
Department of Mathematics\\
Massachusetts Institute of Technology\\
77 Massachusetts Avenue\\
Cambridge, MA 02139\\
USA
}
\email{hkrieger@math.mit.edu}

\author{K. Nguyen}
\address{
Khoa Nguyen\\
Department of Mathematics\\
University of British Columbia\\
Vancouver, BC V6T 1Z2\\
Canada}
\email{dknguyen@math.ubc.ca}

\author{H. Ye}
\address{
Hexi Ye\\
Department of Mathematics\\
University of British Columbia\\
Vancouver, BC V6T 1Z2\\
Canada}
\email{yehexi@math.ubc.ca}

\begin{abstract}
We establish the equidistribution with respect to the bifurcation measure of post-critically finite maps in any one-dimensional algebraic family of unicritical polynomials. Using this equidistribution result, together with a combinatorial analysis of certain algebraic correspondences on the complement of the Mandelbrot set $\cM_2$ (or generalized Mandelbrot set $\cM_d$ for degree $d>2$), we classify all algebraic curves $C\subset \C^2$ with Zariski-dense subsets of points $(a,b)\in C$, such that both $z^d+a$ and $z^d+b$ are simultaneously postcritically finite for a fixed degree $d\geq 2$. Our result is analogous to the famous result of Andr\'e \cite{Andre} regarding plane curves which contain infinitely many points with both coordinates CM parameters, and is the first complete case of the dynamical Andr\'e-Oort phenomenon studied by Baker and DeMarco \cite{Baker-DeMarco}. 

\end{abstract}


\thanks{The research of H.K. was partially supported by an NSF grant, while the research of D.G., K.N. and H.Y. was partially supported by NSERC grants.}

\maketitle


\section{Introduction}

In the past 20 years there was a considerable interest in studying the principle of \emph{unlikely intersections} in arithmetic geometry (for a comprehensive discussion, see the book of Zannier \cite{Zannier-book}). Informally, this principle of unlikely intersections (of which special cases are both the Andr\'e-Oort and the Pink-Zilber conjectures) predicts that each time an intersection of an algebraic variety with a family of algebraic varieties is larger than expected, then this is explained by the presence of a rigid geometric constraint. Motivated by a version of the Pink-Zilber Conjecture for semiabelian schemes, Masser and Zannier (see \cite{M-Z-1, M-Z-2}) proved that in a non-constant elliptic family $E_t$ parametrized by $t\in \C$, for any two sections $\{P_t\}_t$ and $\{Q_t\}_t$, if there exist infinitely many $t\in \C$ such that both $P_t$ and $Q_t$ are torsion points on $E_t$, then the two sections are linearly dependent. 

Motivated by a question of Zannier, Baker and DeMarco \cite{Baker-DeMarco} proved a first result for the unlikely intersections principle in the context of algebraic dynamics. More precisely, Baker and DeMarco showed that for an integer $d\ge 2$, and for two complex numbers $a$ and $b$, if there exist infinitely many $t\in\C$ such that both $a$ and $b$ are preperiodic under the action of $z\mapsto z^d+t$, then $a^d=b^d$. Baker and DeMarco's result \cite{Baker-DeMarco} can be seen as an analogue of Masser and Zannier result \cite{M-Z-1, M-Z-2} (which can be reformulated for simultaneous preperiodic points in a family of Latt\'es maps) \emph{without} the presence of an algebraic group. The absence of an algebraic group in the background is an added difficulty for the problem, which is solved by Baker and DeMarco employing an argument which relies on a theorem regarding the equidistribution of points of small height for algebraic dynamical systems (see \cite{Baker-Rumely06, CLoir, favre-rivera06}). New results followed (see \cite{GHT-ANT, Matt-Laura-2}) extending the results of \cite{Baker-DeMarco} to  arbitrary $1$-parameter families of polynomials. 

In \cite{Matt-Laura-2}, Baker and DeMarco posed a very general question for families of dynamical systems, which is motivated by the classical Andr\'e-Oort. As a parallel to the classical Andr\'e-Oort Conjecture, Baker and DeMarco's question asks that if a subvariety $V$ of the moduli space of rational maps of given degree contains a Zariski dense set of \emph{special points}, then $V$ itself is \emph{special} (i.e., cut out by critical orbit relations; see \cite{Matt-Laura-2} for more details). The \emph{special points} of the moduli space in Baker-DeMarco's question are the ones corresponding to \emph{postcritically finite (PCF) maps} $f$, i.e. each critical point of $f$ is preperiodic. In this article we prove a general result for curves supporting this Dynamical Andr\'e-Oort Conjecture.

\begin{theorem}\label{main theorem}
Let $C$ be an irreducible algebraic plane curve defined over $\C$, and let $d\geq 2$ be an integer. There exist infinitely many points $(a,b)\in C$, such that both $z\mapsto z^d+a$ and $z\mapsto z^d+b$ are postcritically finite, if and only if  one of the following conditions holds:
\begin{itemize}
\item[(1)] there exists $t_0\in \C$ such that $z\mapsto z^d+t_0$ is PCF and $C$ is the curve $\{t_0\}\times \A^1$; 
\item[(2)] there exists $t_0\in \C$ such that $z\mapsto z^d+t_0$ is PCF and $C$ is the curve $\A^1\times \{t_0\}$; 
\item[(3)] there exists a $(d-1)$-st root of unity $\zeta$ such that $C$ is the zero locus of the equation $y-\zeta x=0$.
\end{itemize}
\end{theorem}

In \cite{GKN:preprint}, Theorem~\ref{main theorem} was proven in the special case $C$ is the graph of a polynomial. The extension to arbitrary curves in Theorem~\ref{main theorem} requires overcoming several technical difficulties. We also note that Theorem~\ref{main theorem} can be viewed as a dynamical analogue of Andr\'e's theorem \cite{Andre} regarding plane curves containing infinitely many points with both coordinates CM points in the parameter space of elliptic curves. In the world of polynomial dynamics, the equivalent notion of a CM elliptic curve is a PCF polynomial. Indeed,  the parallel between the two can be viewed also at the level of arboreal Galois representation associated to a polynomial which is expected to have smaller image for PCF maps, 
analogous to the situation for the  Galois representation associated to an elliptic curve, which  has smaller image in the case of CM elliptic curves; for more details, see \cite{Jones, Pink-1, Pink-2, Pink-3}.

We observe that it is immediate to see that a curve of the form (1)~to~(3) as in the conclusion of Theorem~\ref{main theorem} contains infinitely many points with both coordinates PCF parameters; the difficulty in Theorem~\ref{main theorem} is proving that \emph{only} such curves have infinitely many such points. If $C$ does not project dominantly onto one of the axis of $\A^2$, it is immediate to see that $C$ must have the form (1)~or~(2) above. So, the content of Theorem~\ref{main theorem} is to show that when $C$ projects dominantly onto both axis of $\A^2$, and in addition $C$ contains infinitely many points with both coordinates PCF parameters, then $C$ must be of the form (3) as in the conclusion of our result. 

Theorem~\ref{main theorem} can be viewed as a generalization of the problem studied in \cite{Baker-DeMarco}, as follows. Given a plane curve $C$, we have two families of polynomials parametrized by the points $t\in C$: $\bff_{1,t}(z):=z^d+\pi_1(t)$ and $\bff_{2,t}(z):=z^d+\pi_2(t)$, where $\pi_1$ and $\pi_2$ are the two projections of $C$ onto the two coordinate axis of the affine plane. Then we study under what conditions there are infinitely many $t\in C$ such that $0$ is preperiodic under both $\bff_{1,t}$ and $\bff_{2,t}$. More generally, one could consider any two families of rational maps $\bff_{1,t}$ and $\bff_{2,t}$ parametrized by points $t$ on some curve $C$, take any two rational maps $c_1,c_2:C\lra \bP^1$, and ask under what conditions on the curve $C$, on the two families $\bff_1$ and $\bff_2$, and on the starting points $c_1$ and $c_2$, there exist infinitely many points $t\in C$ such that both $c_1(t)$ and $c_2(t)$ are preperiodic under the action of $\bff_{1,t}$, respectively of $\bff_{2,t}$.   However this is a very hard question, and there are only a handful of results with restricted conditions in the literature; see \cite{Baker-DeMarco, GHT-ANT, Matt-Laura-2, GHT:preprint}. For examples in \cite{Baker-DeMarco, GHT-ANT, Matt-Laura-2}, $\bff_1$ and $\bff_2$ are families of polynomials and $C=\A^1$. In  \cite{GHT:preprint}, the family of rational functions, parametrized by a projective curve $C$, must have exactly one degenerate point on $C$ and also the families $\bff_1$ and $\bff_2$ must satisfy additional technical conditions. In this article, we release all the restrictions on the curve $C$, which parametrizes a family of  unicritical polynomials. Also, we note that the result of \cite{GKN:preprint} relied on the results from \cite{GHT-ANT}, hence the restriction to curves $C$ which were graphs of polynomials because then the families of maps $\bff_{1,t}$ and $\bff_{2,t}$ were parametrized by the affine line.

One of the main ingredients of  our article (and also of all of the above articles) is the arithmetic equidistribution of small points on an algebraic variety (in the case of $\P^1$, see \cite{Baker-Rumely06, favre-rivera06}, in the general case of  curves, see \cite{CLoir}, while for arbitrary varieties, see \cite{Yuan}). Another main ingredient of this article is the geometric properties of the generalized Mandelbrot sets $\cM_d$ (recall that $\cM_d$ is the set of all $t\in\C$ where the orbit of $0$ under $z\mapsto z^d+t$ is bounded). More precisely, we use the combinatorial behaviour of the landing of the external rays, from which we get the precise equations for the curves $C$ in Theorem~\ref{main theorem}. Using Yuan's powerful theorem \cite{Yuan} we show that postcritically finite maps equidistribute on the parameter space with respect to the bifurcation measure; see Theorem \ref{pcf equidistribution}. Assuming there exist infinitely many points $(a,b)$ on the plane curve $C$ such that both $z^d+a$ and $z^d+b$ are PCF, then the potential (escape-rate) functions for the bifurcation measures (with respect to the families $z^d+\pi_1(t)$ and $z^d+\pi_2(t)$, where $\pi_1$ and $\pi_2$ are the two projections of $C$ on the coordinates of $\A^2$) are proportional to each other; see Theorem \ref{escape relation}.  Hence we get an algebraic correspondence on the $d$-th generalized Mandelbrot set $\cM_d$: for each $(a,b)\in C$, we have that $a\in \cM_d$ if and only if $b\in \cM_d$. Using the theory of landing external rays on the $d$-th generalized Mandelbrot set we prove that the only algebraic correspondences on $\cM_d$ are linear given by an equation as in the conclusion of Theorem~\ref{main theorem}. 

We are indebted to Laura DeMarco and Thomas Tucker for their careful reading of, and helpful comments on, an early version of this article.  We also thank Bjorn Poonen and Curt McMullen for helpful discussions during the writing of this paper.


\section{Preliminaries}

In this section, we introduce terminologies and results (e.g. Yuan's arithmetic equidistribution theorem \cite{Yuan})  as needed for the latter sections. Though Yuan's equidistribution works for varieties of all dimensions, we focus on the one dimensional case. 

\subsection{The height functions}\label{height subsection} Let $K$ be a number field and $\Kbar$ be the algebraic closure of $K$. The number field $K$ is naturally equipped with a set $\OK$ of pairwise inequivalent nontrivial absolute values, together with positive integers $N_v$ for each $v\in \OK$ such that
\begin{itemize}
\item for each $\alpha \in K^*$, we have $|\alpha|_v=1$ for all but finitely many places $v\in \OK$. 
\item every $\alpha \in K^*$ satisfies the {\em product formula}
\begin{equation}\label{product formula}
   \prod_{v\in \OK} |\alpha|_v^{N_v}=1
   \end{equation}
\end{itemize}
For each $v\in \OK$, let $K_v$ be the completion of $K$ at $v$, let $\Kbar_v$ be the algebraic closure of $K_v$ and let $\C_v$ denote the completion of $\Kbar_v$. We fix an embedding of $\Kbar$ into $\C_v$ for each $v\in \OK$; hence we have a fixed extension of $|\cdot |_v$ on $\Kbar$. When $v$ is archimedean, then $\C_v\cong \C$. For any $x\in \Kbar$, the  Weil height is 
\begin{equation}\label{naive height}
h(x)=\frac{1}{[K(x):K]} \sum_{y\in \Gal(\Kbar/K)\cdot x}~ \sum_{v\in \OK} \log^+|y|_v
\end{equation}
 where $\log^+ z=\log \max\{1, z\}$ for any real number $z$. 

 Let $f\in K[z]$ be any polynomial with degree $d\geq 2$. We use the notation $f^n$ for the composition of $f$ with itself $n$ times. As introduced by Call and Silverman \cite{Call:Silverman}, we have the following {\em canonical height} for every $x\in \Kbar$
 \begin{equation}\label{call-silverman height}
     \hhat_f(x)=\lim_{n\to \infty} \frac{h(f^n(x))}{d^n}
     \end{equation}
 where $h(x)$ is the Weil height from (\ref{naive height}). Call and Silverman \cite{Call:Silverman} showed that the above canonical height is well-defined, and moreover, $\hhat_f(x)\geq 0$ with equality  if and only if $x$ is preperiodic under the iteration of $f$. Hence, $f$ is postcritically finite if and only if all its critical points have canonical height zero.
 
\subsection{Adelic metrized line bundle and equidistribution} 
Let $\cL$ be a line bundle of a nonsingular projective curve $X$ over a number field $K$. As in Subsection~\ref{height subsection} , $K$ is naturally equipped with absolutes $|\cdot|_v$ for $v\in \OK$. A {\em metric} $\|\cdot\|_v$ on $\cL$ is a collection of norms, one for each $x\in X(K_v)$, on the fibres  $\cL(x)$ of the line bundle, with 
   $$\|\alpha s(x)\|_v=|\alpha|_v\|s(x)\|_v$$
for any section $s$ of $\cL$. An {\em adelic metrized line bundle} $\cLbar=\{\cL, \{\|\cdot\|_v\}_{v\in \OK}\}$ over $\cL$ is a collection of metrics on $\cL$, one for each place $v\in \OK$, satisfying certain continuity and coherence conditions; see \cite{Zhang:line, Zhang:metrics}. 

For example, we can define adelic metrized line bundles for $\P^1$ over the line bundle $\cL=\O_{\P^1}(1)$. Let $s=u_0X_0+u_1X_1$ be a global section of $\cL=\O_{\P^1}(1)$, where $u_0$ and $u_1$ are scalars. The metrics are defined for each $[x_0:x_1]\in \P^1(\Kbar)$ as 
   $$\|s\left([x_0: x_1]\right)\|_v:=\frac{|u_0x_0+u_1x_1|_v}{\max\{|x_0|_v, |x_1|_v\}}$$
for places $v\in \OK$. It can be checked without any difficulty that $\cLbar:=\{\cL, \{\|\cdot\|_v\}_{v\in \OK}\}$ defined this way is an adelic metrized line bundle over $\cL$. Moreover, we can work with pullback metrics by an endomorphism of $\P^1$. More precisely, let $F\left([x_0: x_1]\right)=\left(F_0(x_0, x_1): F_1(x_0, x_1)\right)$ be an endomorphism of $\P^1$ where $F_1$ and $F_2$ are coprime homogeneous polynomials of degree $d$. The metrics on  $s=u_0X_0+u_1X_1$ are defined as
   $$\|s\left([x_0:x_1]\right)\|^F_v:=\frac{|u_0x_0+u_1x_1|_v}{\max\{|F_0(x_0, x_1)|_v, |F_1(x_0, x_1)|_v\}^{1/d}}$$
Hence $\cLbar_F:=\{\cL, \{\|\cdot\|^F_v\}_{v\in \OK}\}$ is an adelic metrized line bundle over $\cL$. 

A sequence $\{\cL, \{\|\cdot\|_{v,n}\}_{v\in \OK}\}_{n\geq 1}$ of adelic metrized line bundles over $\cL$ is convergent to $\{\cL, \{\|\cdot\|_v\}_{v\in \OK}\}$, if for all $n$ and all but finitely many $v\in\OK$, $\|\cdot\|_{v,n}=\|\cdot\|_v$, and if $\{\log\frac{\|\cdot\|_{v,n}}{\|\cdot\|_v}\}_{n\geq 1}$ converges to $0$ uniformly on $X(K)$ for all $v\in \OK$. It is clear that the limit $\{\cL, \{\|\cdot\|_v\}_{v\in \OK}\}$ is an adelic metrized line bundle. 

All metrics we consider here are induced by models or uniform limits of metrics from models; see\cite{Zhang:metrics, Yuan}. An adelic metrized line bundle $\cLbar$ is {\em algebraic} if there is a model $\mathcal{X}$ of $X$ that induces the metrics on $\cL$. An algebraic adelic metrized line bundle $\cLbar$ is {\em semipositive} if $\cLbar$ has semipositive curvatures at archimedean places and non-negative degree on any complete vertical curve of $\mathcal{X}$.  An adelic metrized line bundle $\cLbar$ is semipositive if it is the uniform limit of a sequence of algebraic adelic semipositive metrics over $\cL$.

For a semipositive line bundle $\cLbar$, we can define a height for each subvariety $Y$ of $X$ (denoted $\hhat_\cLbar(Y)$); see \cite{Zhang:metrics} for more details. Let $X$ be a nonsingular projective curve. In the case of points on $X$, the height for $x\in X(\Kbar)$ is given by 
\begin{equation}\label{points height}
\hhat_{\cLbar}(x)=\frac{1}{|\Gal(\Kbar/K)\cdot x|}\sum_{y\in\Gal(\Kbar/K)\cdot x}~\sum_{v\in \OK}-N_v\log\|s(y)\|_v
\end{equation}
where $|\Gal(\Kbar/K)\cdot x|$ is the number of points in the Galois orbits of $x$, and $s$ is any meromorphic section of $\cL$ with support disjoint from $\Gal(\Kbar/K)\cdot x$. A sequence of points $x_n\in X(\Kbar)$ is {\em small}, if $\lim_{n\to \infty} \hhat_{\cLbar}(x_n)=\hhat_{\cLbar}(X)$. 

\begin{theorem}\cite[Theorem 3.1]{Yuan}\label{yuan equidistribution}
Suppose $X$ is a projective curve over a number field $K$, and $\cLbar$ is a metrized line bundle over $X$ such that $\cL$ is ample and the metric is semipositive. Let $\{x_n\}$ be a non-repeating sequence of points in $X(\Kbar)$ which is small. Then for any $v\in \OK$, the Galois orbits of the sequence $\{x_n\}$ are equidistributed in the analytic space $X^{an}_{\C_v}$ with respect to the probability measure $d\mu_v=c_1(\cLbar)_v/\deg_\cL(X)$. 
\end{theorem}
\begin{remark}
\label{definition equidistribution remark}
When $v$ is archimedean, $X_{\C_v}^{an}$ corresponds to $X(\C)$ and the curvature $c_1(\cLbar)_v$ of the metric $\|\cdot\|_v$ is given by $c_1(\cLbar)_v=\frac{\partial \overline{\partial}}{\pi i}\log \|\cdot\|_v$. For non-archimedean place $v$, $X_{\C_v}^{an}$ is the Berkovich space associated to $X(\C_v)$, and Chambert-Loir \cite{CLoir} constructed an analog of curvature on $X_{\C_v}^{an}$. The precise meaning of the equidistribution above is: 
   $$\lim_{n\to \infty} \frac{1}{|\Gal(\Kbar/K)\cdot x_n|}\sum_{y\in \Gal(\Kbar/K)\cdot x_n} \delta_{y}=\mu_v$$
where $\delta_{y}$ is point mass probability measure supported on $y\in X^{an}_{\C_v}$, and the limit is the weak limit for probability measures on compact space $X^{an}_{\C_v}$. 
\end{remark}


\section{Equidistribution of PCF points}
\label{equidistribution section}

In Sections~\ref{equidistribution section} and \ref{identical section}, we prove that for a (one dimensional and non-isotrivial) family of unicritical polynomials with degree $d\geq 2$, the set of postcritically finite polynomials equidistributes on the parameter space. The main tool we use in this section is the arithmetic equidistribution theorem introduced in the previous section; for the setup of the equidistribution theorem, we follow \cite{GHT:preprint}. We start by stating Theorem~\ref{pcf equidistribution} which is our main goal; in order to do this we need to set up the proper notation.

\subsection{Statement of the equidistribution theorem for PCF parameters}
\label{subsection PCF statement}

For the definition of algebraic families of unicritical polynomials, we follow \cite{DeMarco: heights}. Let $f: X'\times \C\to \C$ be a one dimensional {\em algebraic family} of unicritical polynomials of degree $d\geq 2$. That is, $X'$ is a Zariski dense, open subset of an irreducible, smooth curve  $X$ defined over $\C$, while $\psi:X'\lra \A^1$ is a morphism, and $f$ is a polynomial map of degree $d$ given by 
  $f_{\psi(t)}(z):=f(t,z)=z^d+\psi(t)$,  for each $t\in X'(\C)$. We say that $f$ is {\em isotrivial} if $\psi$ is a constant map. Since there is nothing to study for an isotrivial family of unicritical polynomials,  we focus on the non-isotrivial case. In addition, we assume $X$ and $X'$ are defined over a number field $K$. If $\psi$ is a morphism defined over $K$, then we call $f$ an algebraic family of unicritical polynomials over the number field $K$. We can view $X$ as a {\em parameter space} for an algebraic family of unicritical polynomials. The main goal for us is proving the following result.
 
\begin{theorem}\label{pcf equidistribution}
Let $f: X'\times \C\to \C$ be a non-isotrivial, one dimensional algebraic family of degree $d\geq 2$ unicritical polynomials over a number field $K$. The set of parameters $t\in X'(\Kbar)$, for which $f(t, z): \C \to \C$ is postcritically finite,  equidistributes on the parameter space $X(\C)$ (with respect to the normalized bifurcation measure). 
\end{theorem}

We postpone the proof of this theorem to Subsection~\ref{curvature and bif} (for a precise definition of the equidistribution, see Remark~\ref{definition equidistribution remark}). In Subsection~\ref{bif subsection} we define the (normalized) biffurcation measure and also show its connection with the measures corresponding to certain adelic metrized line bundles as appearing in the work of Yuan \cite{Yuan}. The key to our proof of Theorem~\ref{pcf equidistribution} is the equidistribution theorem of Yuan (see Theorem~\ref{yuan equidistribution} and its consequence to our setting stated in Theorem~\ref{general equidistriibution}). 

\subsection{Metrics on a line bundle} 
\label{metrics definition subsection}

As previously stated, a non-isotrivial, one dimensional algebraic family of unicritical polynomials over a number field  $K$ is uniquely determined by a morphism $\psi: X'\to \A^1$, where $X'$ is a Zariski dense open subset of an irreducible, nonsingular projective curve $X$ defined over $K$. Hence the morphism $\psi: X'\lra \A^1$ induces a unique morphism $\psi:X\lra \P^1$ (for the sake of simplifying the notation, we use the same notation for both morphisms).
  
Let $\cL$ be the line bundle on the projective curve $X$ which is the pullback of $\O_{\P^1}(1)$ by $\psi$, i.e. $\cL:=\psi^*\O_{\P^1}(1)$. Next we are going to introduce metrics on this line bundle. Let $S$ be the set of poles of $\psi$ on $X$, i.e. $S$ consists of all $x\in X$ such that $\psi(x)=[1:0]$ (the infinity point of $\bP^1$). Let $X_0, X_1$ be the canonical sections on $\P^1$, and $s:=\psi^*(u_0X_0+u_1X_1)$ be a section of the line bundle $L$ with $u_0$ and $u_1$ being the scalars. 
 For any point $t\in X(\C_v)\backslash S$, we define the metrics for each $n\geq 1$ and each place $v\in \OK$ as follows: 
\begin{equation}\label{metric definition}
    \|s(t)\|_{v, n}:=\frac{|u_0\psi(t)+u_1|_v^{1/d_\psi}}{\max\{1, |f_{\psi(t)}^n(0)|_v\}^{1/(d_{\psi}\cdot d^{n-1})}}
\end{equation}
where $d_\psi$ is the degree of the morphism $\psi: X\to \P^1$. Moreover, for each $t_0\in S\subset X(\C_v)$, we define 
\begin{equation}\label{at poles}
\|s(t_0)\|_{v, n}:=v\textup{-$\lim_{t\to t_0}$} \|s(t)\|_{v,n}=|u_0|_v^{1/d_{\psi}}.
\end{equation}
The last equality in the above formula is obvious once we notice that when $t$ is close to $t_0$, we have $|f^n_{\psi(t)}(0)|_v^{1/ d^{n-1}}\sim |\psi(t)|_v$.
\begin{lemma}\label{good reduction}
For any nonarchimedean place $v\in \OK$ and any integer $n\geq 1$, we have 
   $$\|\cdot\|_{v,n}=\|\cdot\|_{v,1}$$
 on the line bundle $\cL$. 
\end{lemma}

\begin{proof} It suffices to show that $\max\{1,|f_{\psi(t)}^n(0)|_v\}=\max\{ 1, |\psi(t)|_v^{d^{n-1}}\}$ for all $t\in X(\C_v)\backslash S$, where $S$ is the set of poles for $\psi$. We prove it by induction. Suppose $\max\{1,|f_{\psi(t)}^n(0)|_v\}=\max\{ 1, |\psi(t)|_v^{d^{n-1}}\}$. If $|\psi(t)|_v\leq 1$, then $|f_{\psi(t)}^n(0)|_v\leq 1$. Hence $|f^{n+1}_{\psi(t)}(0)|_v=|(f^{n}_{\psi(t)}(0))^d+\psi(t)|_v\leq \max\{|(f^{n}_{\psi(t)}(0))^d|_v, |\psi(t)|_v\}\leq 1$ as $v$ is nonarchimedean. Otherwise if $|\psi(t)|_v> 1$ and $|f_{\psi(t)}^n(0)|_v=|\psi(t)|_v^{d^{n-1}}\geq |\psi(t)|_v>1$, then $|f^{n+1}_{\psi(t)}(0)|_v=|(f^{n}_{\psi(t)}(0))^d+\psi(t)|_v=|(f^{n}_{\psi(t)}(0))^d|_v=|\psi(t)|_v^{d^n}$.
\end{proof}

We define the metric 
$$\|s(t)\|_v:=\lim_{n\to \infty} \|s(t)\|_{v,n}\text{ for each place }v$$
and we prove next that $\log\|\cdot \|_{v,n}$ converges uniformly to $\log\|\cdot \|_v$.

\begin{prop}\label{convergence prop}
For each place $v\in \OK$,  $\log\|\cdot\|_{v,n}$ converges uniformly on $X(\C_v)$ to $\log\|\cdot\|_v$. 
\end{prop}

\begin{proof} Fix a place $v\in \OK$ and a real number $R$ greater than $3$.  Let 
  $$R_{v,n}(t):=\max\{ 1, |f_{\psi(t)}^n(0)|_v\}^{1/d^{n-1}}.$$
To prove this proposition, it suffices to show that $\{\log R_{v,n}(t)\}_{n\geq 1}$, as a sequence of functions on $X(\C_v)$, converges uniformly.  

First we prove the uniform convergence assuming $|\psi(t)|_v\leq R$. Then from the definition of $R_{v,n}(t)$, we know that

\begin{equation*}
\begin{split}
R_{v, n+1}^{d^{n}}(t)& \leq (R_{v,n}^{d^{n-1}}(t))^d+R\\
 & \leq (R+1)\cdot R_{v,n}^{d^{n-1}}(t) \textup{, since $R_{v,n}(t)\geq 1$}\\
 &\leq  2R\cdot R_{v,n}^{d^{n}}(t)
\end{split}
\end{equation*}
Similarly, if $R_{v,n}^{d^{n}}(t)\geq 2R$, then $R_{v, n+1}^{d^{n}}(t)\geq (R_{v, n}^{d^{n-1}}(t))^d-R\geq R_{v, n}^{d^{n}}(t)/2\geq R_{v, n}^{d^{n}}(t)/2R$. On the other hand, if $R_{v,n}^{d^{n}}(t)< 2R$, then $R_{v, n+1}^{d^{n}}(t) \geq R_{v, n}^{d^{n}}(t)/2R$. So in all cases, 
\begin{equation}
\label{R 0}   
\frac{1}{(2R)^{1/d^{n-1}}}\leq \frac{R_{v,n+1}(t)}{R_{v,n}(t)}\leq (2R)^{1/d^{n-1}}
\end{equation}
   which yields the uniform convergence of $\{\log R_{v,n}(t)\}_n$ (by taking logarithms in \eqref{R 0} and the use a telescoping sum) for all $t\in X(\C_v)$ satisfying $|\psi(t)|_v\leq R$. 

Secondly, we assume $t\in X(\C_v)\backslash S$ such that $|\psi(t)|_v> R$. We prove by induction on $n$ that $R^{d^{n-1}}_{v,n}(t)\geq |\psi(t)|_v$. Indeed, the case $n=1$ is obvious, while in general (also noting that $R>3$ and $d\ge 2$) we have: 
   $$R^{d^{n}}_{v,n+1}(t)\geq (R^{d^{n-1}}_{v,n}(t))^d-|\psi(t)|_v\geq R^{d^{n}}_{v,n}(t)-\frac{R^{d^{n}}_{v,n}(t)}{2}\geq \frac{R^{d^{n}}_{v,n}(t)}{2}\geq R^{d^{n-1}}_{v,n}(t)\geq  |\psi(t)|_v.$$
Then it is easy to see that 
   $$|R_{v,n+1}^{d^{n}}(t)-(R^{d^{n-1}}_{v,n}(t))^d|\leq |\psi(t)|_v\leq \frac{R^{d^{n}}_{v,n}(t)}{2}$$
and so, 
   $$\left|\frac{R_{v,n+1}^{d^{n}}(t)}{R^{d^{n}}_{v,n}(t)}-1\right|\leq \frac{1}{2}$$
or equivalently, 
\begin{equation}
\label{R 1}   
\left(\frac{1}{2}\right)^{1/d^{n}}\leq \frac{R_{v,n+1}(t)}{R_{v,n}(t)}\leq \left(\frac{3}{2}\right)^{1/d^{n}}.
\end{equation}
  
Taking logarithms in \eqref{R 1} and using again a telescoping sum, we obtain the uniform convergence of $\{\log R_{v,n}(t)\}_n$ for all $t\in X(\C_v)\setminus S$ such that $|\psi(t)|_v>R$. Finally, using also the convergence at the poles (according to \eqref{at poles}), we conclude the proof of Proposition~\ref{convergence prop}. 
\end{proof} 

\subsection{Equidistribution of small points} 
We use the same construction as in \cite[Section~7]{GHT:preprint}. So, from Lemma \ref{good reduction} and Proposition \ref{convergence prop}, we know that 
\begin{equation}\label{definition of ALB}\cLbar:=(\cL, \{\|\cdot\|_v\}_{v\in \OK}),
\end{equation}
is an adelic metrized line bundle which is semipositive.  The height function $\hhat_\cLbar$ on $X(\Kbar)$ associated to $\cLbar$ is given by: 
\begin{equation}\label{height of line bundle}
\hhat_{\cLbar}(t):=\sum_{v\in \OK}\frac{N_v}{|\Gal(\Kbar/K)\cdot t|}\cdot \sum_{y\in\Gal(\Kbar/K)\cdot t}-\log \|s(y)\|_v, \textup{ for any $t\in X(\Kbar)$}
\end{equation}
where $s$ is any section of $\cL=\psi^*\O_{\P^1}(1)$ which does not vanish on the Galois orbits of $t$. The product formula guarantees that this height does not depend on the section $s$ in the above formula. 
    
The adelic metrized line bundle $\cLbar$ is uniquely determined by the non-constant morphism $\psi: X\to \P^1$ (defined over $K$). For convenience, we use a new notation for the height  $\hhat_\cLbar$ on $X$ associated to the morphism $\psi$:
\begin{equation}\label{psi height}
\hhat_{\psi}(t):=\hhat_\cLbar(t), \textup{ for $t\in X(\Kbar)$}.
\end{equation}

So, as a corollary of Theorem~\ref{yuan equidistribution} applied to the problem we study, we obtain the following equidistribution theorem for points of height tending to $0$. 
\begin{theorem}\label{general equidistriibution}
Let $X$ be a nonsingular projective curve over a number field $K$ and $\psi: X\to \P^1$ be a non-constant morphism defined over $K$. The adelic metrized line bundle $\cLbar$ in (\ref{definition of ALB}), corresponding to the ample line bundle $\cL=\psi^*\O_{\P^1}(1)$ is semipositive. Let $\{t_n\}_{n\geq 1}\subset X(\Kbar)$ be any non-repeating sequence of small points, i.e. $\lim_{n\to \infty}\hhat_\psi(t_n)=0$. Then for any place $v\in \OK$, the Galois orbits of this sequence are equidistributed in the analytic space $X^{an}_{\C_v}$ with respect to the probability measure $ d \mu_v=c_1(\cLbar)_v/\deg_{\cL}(X)$. 
\end{theorem}

\begin{remark}
We note that $h_\cLbar(X)=0$ because $X$ contains an infinite set of points with height $0$ (see \cite[Theorem~(1.10)]{Zhang:metrics}, Proposition \ref{height relations} and Remark \ref{zariski dense}).
\end{remark}

We obtain next the relation between the two heights $\hhat_\psi$ and $\hhat_{f_{\psi(t)}}$; we recall that $\hhat_{f_{\psi(t)}}$ is the canonical height for points on the affine line under the action of the polynomial $$f_{\psi(t)}(z):=z^d+\psi(t).$$  Also, we recall that $S$ is the set of poles for $\psi$. 
\begin{prop}\label{height relations}
For each $t\in X(\Kbar)\setminus S$, we have $\hhat_\psi(t)=\frac{d}{d_\psi}\cdot \hhat_{f_{\psi(t)}}(0)$, while  $\hhat_\psi (t)=0$ for each $t\in S$. In particular, $\hhat_\psi(t)\geq 0$ on $X(\Kbar)$ with equality  if and only if $t$ is a pole of $\psi$ or $f_{\psi(t)}$ is postcritically finite. 
\end{prop}

\begin{proof} First, assume that $t\in X(\Kbar)$ is a pole of $\psi$, i.e. $t\in S$. As $\psi$ is defined over $K$, then the Galois orbit of $t$ is contained in  $S$. By the product formula \eqref{product formula} together with the definition of metrics at a pole \eqref{at poles} and the definition of the height \eqref{height of line bundle}, we see that $\hhat_\psi (t)=0$. 

Secondly, let $t\in X(\Kbar)\backslash S$; in this case, the points in the  Galois orbit of $t$ are not poles of $\psi$. Let $X_0, X_1$ be the two canonical sections of $\O_{\P^1}(1)$, and pick $u_0, u_1\in K$ such that the section $u_0X_0+u_1X_1$ of  $\O_{\P^1}(1)$ does not vanish on $[\psi(t):1]\in \bP^1(\Kbar)$. For each $y\in \Gal(\Kbar/K)\cdot t$, this section does not vanish on $\psi(y)$. Define $s:=\psi^*(u_0X_0+u_1X_1)$, noting that $s$ does not vanish on the Galois orbits of $t\in X(\Kbar)$. Writing $d_\psi:=\deg(\psi)$, we have 

\begin{equation*}
\begin{split}
\hhat_\psi (t)&= \sum_{v\in \OK} \sum_{y\in\Gal(\Kbar/K)\cdot t}\frac{-N_v\cdot \log \|s(y)\|_v}{|\Gal(\Kbar/K)\cdot t|}, \textup{ by (\ref{height of line bundle}) and (\ref{psi height})}\\
 &= \sum_{v\in \OK} \sum_{y\in\Gal(\Kbar/K)\cdot t}\lim_{n\to \infty}\frac{N_v\cdot \log \max\{1, |f_{\psi(y)}^n(0)|_v\}^{1/(d_\psi\cdot d^{n-1})}}{|\Gal(\Kbar/K)\cdot t|}, \textup{ by (\ref{metric definition}) and (\ref{product formula})}\\
 &=  \frac{1}{|\Gal(\Kbar/K)\cdot t|} \lim_{n\to \infty}\sum_{v\in \OK} \sum_{y\in\Gal(\Kbar/K)\cdot t}\frac{N_v\cdot \log^+ |f_{\psi(y)}^n(0)|_v}{d_\psi\cdot d^{n-1}}\\
 &=\frac{d}{d_\psi}\cdot \hhat_{f_{\psi(t)}}(0), \textup{ by \eqref{naive height} and \eqref{call-silverman height}.}
\end{split}
\end{equation*}
The second part of the proposition follows, since $\hhat_{f_{\psi(t)}}(0)\geq 0$ with equality if and only if the critical point $0$ is preperiodic under iteration of $f_{\psi(t)}$; see \cite{Call:Silverman}. 
\end{proof}

\begin{remark}\label{zariski dense}
It is well known that there are infinitely many $t\in \Qbar$ such that $f_t(z)=z^d+t$ is postcritically finite. Since the morphism $\psi: X\to \P^1$ is non-constant, the set of points with zero height (for the height function $\hhat_\psi$)  is Zariski dense on $X(\Kbar)$.
\end{remark}


\section{Bifurcation and potential functions}\label{identical section}

In this section, we study the bifurcation of algebraic families of unicritical polynomials, parametrized by quasi-projective curves. Let $X'$ be (as in the previous Section) a Zariski open dense subset of an irreducible, nonsingular, projective curve $X$  which is a parameter space for two families of unicritical polynomials. In Theorem~\ref{escape relation} we prove that if there are infinitely many points in $X'$ such that the corresponding two polynomials for these two families are simultaneously PCF, then these two families of polynomials  have the same normalized bifurcation measure on $X'(\C)$; this result is a consequence of Theorem~\ref{general equidistriibution} and the definition of the bifurcation measure (see Subsection~\ref{curvature and bif}). 

\subsection{Bifurcation} \label{bif subsection} For a holomorphic family $f(t,\cdot): \P^1\to \P^1$ of rational functions of degree $d\geq 2$ parametrized by a complex manifold, we have a stable region, a bifurcation locus (which is the complement of the stable region) and a bifurcation measure (or (1,1)-current) on the parameter space; see \cite{D:current, D:lyap, Dujardin:Favre:critical, McMullen:CDR, Mane:Sad:Sullivan}.  One of the main goals in complex dynamics is to study the stability of holomorphic families (or moduli spaces) of rational functions. In this article, we restrict our study to algebraic families of unicritical polynomials, parametrized by quasi-projective curves. 

We work with the notation as in Subsection~\ref{subsection PCF statement}. So, $X$ is a smooth, irreducible curve, $X'$ is a Zariski dense open subset of $X$, and   
 $f:X'\times \C\to \C$ is an algebraic family of unicritical polynomials of  degree $d\geq 2$, i.e. $f_{\psi(t)}(z)=z^d+\psi(t)$ where $\psi: X'\lra \A^1$ is a morphism. A point $t_0\in X'$ is {\em stable} if the Julia sets $J_{f_{\psi(t)}}$ are moving holomorphically  in a neighbourhood of $t_0$, or equivalently,   $\{f^n_{\psi(t)}(0)\}_{n\geq 1}$ 
is a normal family of functions on some neighbourhood of $t_0$. The {\em bifurcation locus} on $X'(\C)$ is the set of parameters where $f_{\psi(t)}$ fails to be stable. By definition, the stable region is always an open subset of $X'(\C)$. 

We define the {\em escape-rate function} for $\psi$ as 
   $$G_{\psi}(t):=\lim_{n\to \infty}\frac{1}{d^n}\log^+|f_{\psi(t)}^n(0)|,$$
which is a subharmonic function on $X'(\C)$. It is convenient to extend the function $G_\psi$ on $X(\C)$ by defining it $$G_\psi(t)=0\text{ for each }t\in (X\setminus X')(\C).$$ The differential of the {\em bifurcation measure} is defined as 
\begin{equation}\label{definition of bifurcation measure}
d \mu_{\psi}:=dd^c G_{\psi}(t)
\end{equation}
with $dd^c=\frac{i}{\pi }{\partial \overline{\partial}}$ being the Laplacian operator. For the sake of simplifying the notation, when $\psi(t)=t$ is the identity map, we use $\mu$ and $G(t)$ instead of $\mu_\psi$ and $G_\psi(t)$. The support of the bifurcation measure coincides with the bifurcation locus on $X'(\C)$, and the bifurcation locus is empty if and only if $\psi$ is a constant  (i.e. $f_{\psi(t)}$ is isotrivial). From the definition of the the escape-rate function, we see that 
   $$G_\psi(t)=G(\psi(t)), $$
i.e. $G_\psi=\psi^* G$ is the pullback of the escape-rate function on the complex plane by $\psi$. Hence the bifurcation measure (resp. bifurcation locus) is the pullback of the bifurcation measure (resp. bifurcation locus) on the complex plane
   $$\mu_\psi=\psi^*\mu,$$
i.e. $\mu_\psi(A)=\mu(\psi(A))$ for $A\subset X'(\C)$ with $\psi$ being injective on $A$. 
\subsection{The generalized Mandelbrot sets} Here we deal with the simplest case: $\psi(t)=t$ (i.e. $f_{\psi(t)}(z)=f_t(z)=z^d+t$) and $X'$ itself is the affine (complex) line. The degree $d$ generalized Mandelbrot set $\cM_d$ is the set of parameters where the critical point $0$ is bounded under the iterates of $f_t$
   $$\cM_d:=\{t\in \C: ~ |f^n_t(0)| \not \to \infty \textup{ as $n\to \infty$} \}$$
When $d=2$, $\cM_2$ is the classical Mandelbrot set. See Figure \ref{external rays} for  the pictures of $\cM_2$ and $\cM_3$. We recall some basic properties of the generalized Mandelbrot sets. Every generalized Mandelbrot set is bounded and simply connected, and there is a unique biholomorphic map $\Phi$ (depending on $d$) from $\C\backslash \cM_d$ to the complement of the closed unit disk $\C\backslash \overline{\D}$ 
\begin{equation}\label{change coordinate}
\Phi:  \C\backslash \cM_d \tilde{\longrightarrow} \C\backslash \overline{\D}  
\end{equation}
with $\Phi(t)=t+O(1), \textup{ for $|t|>>0$}$. The Green's function $G_{\cM_d}$ for the compact set ${\cM_d}$ on $\C\backslash\cM_d$ is given by 
\begin{equation}\label{Phi}
G_{\cM_d}(t)=\log |\Phi(t)|
\end{equation}
and it is known that $G_{\cM_d}(t)=d\cdot G(t)$ (for example, see \cite{Baker-DeMarco}). Moreover, the escape-rate function satisfies the inequality $G(t)\geq 0$ with equality if and only if $t\in \cM_d$. The bifurcation locus for $f_t$ is the boundary $\partial \cM_d$ of $\cM_d$, and the bifurcation measure is proportional to the harmonic measure for $\cM_d$. 
\subsection{Two algebraic families of unicritical polynomials}\label{curvature and bif} Let $X$ be a nonsingular projective curve defined over a number field $K$, let $\psi:X\lra \P^1$ be a non-constant morphism, and let $S\subset X$ be its set of poles. We proceed as in subsection~\ref{metrics definition subsection} and define the adelic metrized line bundle $\cLbar$ endowed with metrics $\|\cdot \|_v$ for each $v\in\Omega_K$. We recall that when $v$ is archimedean, $\C_v\cong \C$ and $X^{an}_{\C_v}\cong X(\C)$.  The curvature $c_1(\cLbar)_v$ of $\|\cdot\|_v$  is given by $c_1(\cLbar)_v=-dd^c\log \|\cdot\|_v$. 
 
For the rest of this subsection, we fix an archimedean place $v$ and identify $\C_v$ with $\C$. For $t_0\in X(\C)\backslash S$, we let $s$ be a section on $\bP^1$ defined over $K$ which does not vanish at $\psi(t_0)$. Hence for $t\in X(\C)$ in a neighbourhood of $t_0$, using \eqref{metric definition},  we have 
   $$c_1(\cLbar)_v(t)=-dd^c\log \|\psi^*(s)(t_0)\|_v=dd^c \lim_{n\to \infty}\frac{\log^+|f_{\psi(t)}^n(0)|_v}{d_\psi \cdot d^{n-1}}=\frac{d}{d_\psi}\cdot dd^cG_\psi(t).$$
For the bifurcation measure $\mu_\psi$ on $X(\C)\backslash S$, we have 
\begin{equation}\label{measures relations}
\mu_\psi=\frac{d_\psi}{d}\cdot \mu_v
\end{equation}
where $d_\psi$ is the degree of $\psi$. In particular, we consider $\mu_v$ be the \emph{normalized bifurcation measure} with respect to which we get the equidistribution statement from Theorem~\ref{pcf equidistribution}. 

\begin{proof}[Proof of Theorem~\ref{pcf equidistribution}.]
Let $\{t_n\}\subset X'(\C)$ be a sequence of PCF parameters for the algebraic  family $f:X'\times \C\lra \C$ of unicritical polynomials of degree $d$. First of all, we note that each $\psi(t_n)\in\Qbar$ since $z^d+\psi(t_n)$ is a PCF map; since $\psi$ is defined over $\Qbar$, then also $t_n\in\Qbar$.  Then by Proposition~\ref{height relations}, $$\hhat_\psi(t_n)=0=\hhat_{f_{\psi(t_n)}}(0).$$ Using Theorem~\ref{general equidistriibution}, we conclude that the points $\{t_n\}$ equidistribute with respect to $\mu_v$, as desired.   
\end{proof}

Now, we consider two non-constant morphisms $\psi_i: X\to \P^1$ for $i=1, 2$, with sets of poles $S_1$ and $ S_2$, respectively. They determine two algebraic families of unicritical polynomials $f_{\psi_1(t)}$ and $f_{\psi_2(t)}$ of  degree $d\geq 2$. 
\begin{theorem}\label{escape relation}
Suppose there are infinitely many $t\in X\backslash (S_1\cup S_2)$, such that $f_{\psi_1(t)}$ and $f_{\psi_2(t)}$ are simultaneously postcritically finite. Then $d_{\psi_2}\cdot  \mu_{\psi_1}=d_{\psi_1}\cdot \mu_{\psi_2}$ on $X\backslash (S_1\cup S_2)$. Furthermore, on $X(\C)$
 \begin{equation}\label{G1 and G2}
 d_{\psi_2}\cdot  G_{\psi_1}(t)=d_{\psi_1}\cdot  G_{\psi_2}(t).
 \end{equation}
\end{theorem}

\begin{proof}  The relations of the two bifurcation measures is clear from Theorem \ref{general equidistriibution} and (\ref{measures relations}). And then these two families have the same stable region on $X(\C)\backslash (S_1\cup S_2)$.  Let 
  $$H(t):=d_{\psi_2}\cdot  G_{\psi_1}(t)-d_{\psi_1}\cdot  G_{\psi_2}(t)$$
  be the difference of the two continuous subharmonic functions on $X(\C)\backslash (S_1\cup S_2)$. Since $d_{\psi_2}\cdot  \mu_{\psi_1}=d_{\psi_1}\cdot \mu_{\psi_2}$, and also using (\ref{definition of bifurcation measure}), $H(t)$ is harmonic on $X(\C)\backslash (S_1\cup S_2)$. The pullback (by $\psi_1$ or $\psi_2$) of a connected component of the stable region $\C\backslash \partial \cM_d$, consists of finitely many (up to the degree $d_{\psi_1}$ or $d_{\psi_2}$) connected components of the stable region on $X$.  As $\C\backslash \partial \cM_d$ consists of infinitely many connected components (see Figure \ref{external rays}), so is $\psi_2^{-1}(\C\backslash \partial \cM_d)$. Then we can pick one connected component of the stable region on $X(\C)\backslash (S_1\cup S_2)$, such that its images under both $\psi_1$ and $\psi_2$ are stable subsets within the generalized Mandelbrot set $\cM_d$. Hence for any $t$ in this component, $G_{\psi_1}(t)=G(\psi_1(t))=0=G(\psi_2(t))=G_{\psi_2}(t)$, which yields that $H(t)=0$. So the harmonic function $H(t)$ on  $X(\C)\backslash (S_1\cup S_2)$, which is identically zero on some open subset of $X(\C)\backslash (S_1\cup S_2)$, must be zero everywhere, and so \eqref{G1 and G2} follows. 
\end{proof}
   
\begin{remark} \label{phiequality} The set $S$ of poles for $\psi$ is the set of points $t_0\in X$ such that $\lim_{t \to t_0}G_\psi(t)=\infty$. With the same assumptions as in Theorem \ref{escape relation}, one has $S_1=S_2$ for the sets of poles of $\psi_1$ and $ \psi_2$. And moreover, by (\ref{Phi}) (\ref{G1 and G2}), for any $t\in X(\C)$ with $\psi_1(t)\in \C\backslash \cM_d$ (hence $\psi_2(t)\in \C\backslash \cM_d$ by proportionality of $G_{\psi_1}$ and $G_{\psi_2}$), we have 
   $$|\Phi(\psi_1(t))|^{d_{\psi_2}}=|\Phi(\psi_2(t))|^{d_{\psi_1}}.$$
\end{remark}


\section{Proof of the main theorem}

Suppose now that $X$ is an irreducible, nonsingular projective curve satisfying the hypothesis of Theorem \ref{escape relation}.  By Remark \ref{phiequality}, we conclude that for all $ t \in X(\C)$, 
$$\psi_1(t) \in \C \setminus \cM_d \Leftrightarrow \psi_2(t) \in \C \setminus \cM_d,$$ 
and further that the uniformizing map $\Phi: \C \setminus \cM_d \rightarrow \C \setminus \bar{\mathbb{D}}$ satisfies 
$$|\Phi(\psi_1(t))|^{d_{\psi_2}}=|\Phi(\psi_2(t))|^{d_{\psi_1}}.$$

Write $d_1 = d_{\psi_1}, d_2 = d_{\psi_2}$.  Let $X_0$ be a connected, unbounded component of the stable region in $X$; i.e., $X_0$ is a component of the preimage of $\mathbb{C} \setminus \cM_d$ under $\psi_2$. The quotient
$$\Phi(\psi_1(t))^{d_{2}}/ \Phi(\psi_2(t))^{d_{1}}$$
provides a holomorphic map $X_0 \rightarrow S^1$ (where $S^1$ is the complex unit circle); by the Open Mapping Theorem, this map is constant, so there exists $\eta \in \mathbb{R}$ such that for all $t \in X_0,$
\begin{equation} \label{phirelation}
\Phi(\psi_1(t))^{d_{2}} = e^{2 \pi i \eta} \cdot \Phi(\psi_2(t))^{d_{1}}.
\end{equation}

Following \cite{Douady-Hubbard}, the standard tool for studying the behavior of the degree $d$ Mandelbrot set $\cM_d$ is given by the {\em external rays} of the map $\Phi$.  We define the external ray for an angle $\theta \in \mathbb{R} / \mathbb{Z}$ to be 
$$\mathcal{R}(\theta) := \Phi^{-1} (\{ re^{2 \pi i \theta} : r >1 \}).$$
We recall some standard facts about external rays; see Chapters 8 and 13 of \cite{Douady-Hubbard}, and \cite{EMS}.  An external ray $\mathcal{R}(\theta)$ is said to be {\em rational} if $\theta$ is rational.  A point $c \in \cM$ is {\em Misiurewicz} if the critical point $0$ of  $z^d+c$ is strictly preperiodic, and clearly every PCF point on the boundary of $\cM_d$ is a Misiurewicz point. 

\begin{prop} \cite{Douady-Hubbard} All rational rays {\em land}; that is, there exists a unique point $c_{\theta} \in \partial \cM_d$ such that $\lim_{r \rightarrow 1} \Phi^{-1}(re^{2 \pi i \theta}) = c_{\theta}$.  Misiurewicz points are contained in the boundary of $\cM_d$, and every Misiurewicz point is the landing point of at least one rational ray.
\end{prop}

Let $\alpha$ be a periodic point of $f_c(z)=z^d+c$ with exact period $n$. The {\em multiplier} of this cycle is $\lambda:=(f^n_c)'(\alpha)$.  The cycle is {\em attracting} if $|\lambda|<1$, {\em repelling} if $|\lambda| > 1$, and {\em parabolic} if $\lambda$ is a root of unity.   A parameter $c$ is {\em parabolic} if $f_c(z)=z^d+c$ contains a parabolic cycle; in this case, there is a unique parabolic cycle. Parabolic points also lie in $\partial \cM_d$, and are the landing points of rational rays.

\begin{prop} \cite{Douady-Hubbard}, \cite{EMS} Every parabolic point $c$ is the landing point of either one or two rational rays.  If the parabolic cycle of $f_c(z)$ has multiplier $\lambda \ne 1$, then exactly two distinct rational rays land at $c$.
\end{prop}

If $\mathcal{R}(\theta)$ and $\mathcal{R}(\theta')$ land at the same point, we say $\theta$ and $\theta'$ are a {\em landing pair}.  If their common landing point $c$ is parabolic with multiplier $\ne 1$, then $\cM_d \setminus \{ c \}$ consists of two connected components. In this case, the component which does not contain 0 is the {\em wake} $w_c$ of $c$, and if $\mathcal{R}(\theta)$ and $\mathcal{R}(\theta')$ land at $c$, the {\em width} of the wake $w_c$ is defined to be $|w_c| := \theta' - \theta$, assuming $0 < \theta < \theta' < 1$.  For more about external rays, one can refer to \cite{EMS}.  For illustration, see Figure \ref{external rays}.

Recall that a stable, connected component $H$ in $\cM_d$ is hyperbolic of period $\ell$ if $z^d+c$ has an attracting cycle of exact period $\ell$ for every $c\in H$.

\begin{prop} \label{standard pair} For all $k \geq 1,$ $\frac{1}{d^k-1}$ and $\frac{d}{d^k-1}$ are a landing pair, and their landing point $c_k$ lies on the boundary of both the unique period 1 hyperbolic component, and a component of period $k$.  
\end{prop}

The proof of the proposition is by standard arguments; see Proposition 3.5 of \cite{GKN:preprint} for details.

\begin{figure} 
\includegraphics[width=2.25in]{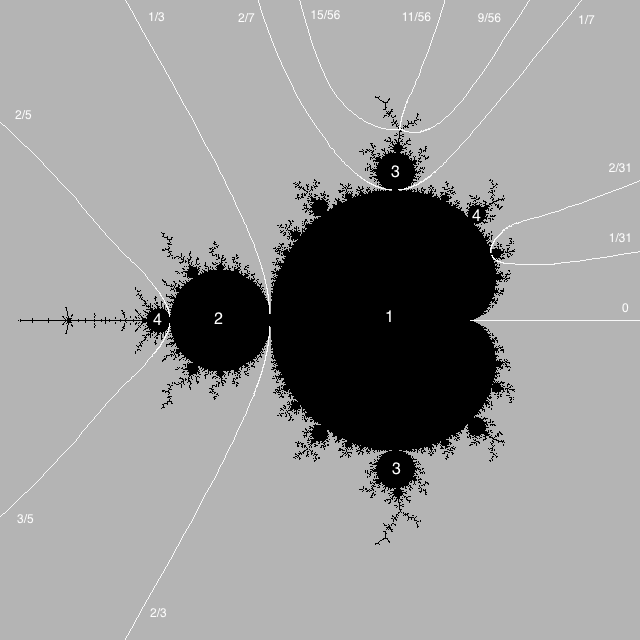}
\includegraphics[width=2.25in]{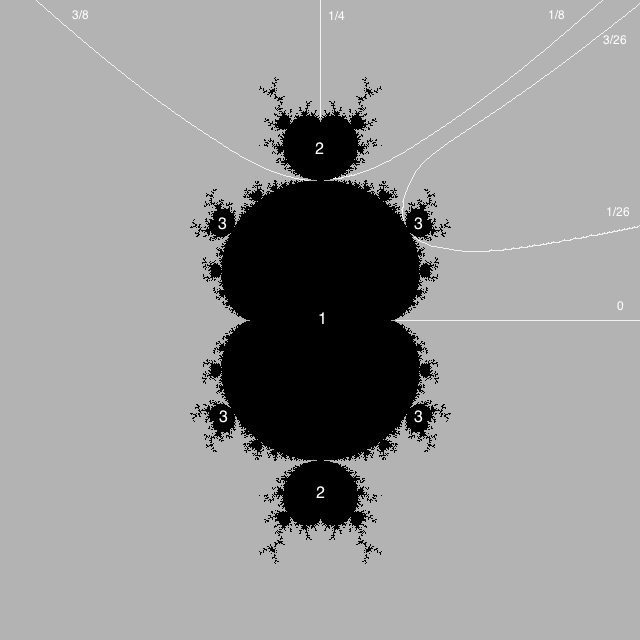}

\caption{ \small Selected external rays of the degree 2 (left) and degree 3 (right) Mandelbrot sets.  Angles of each ray are indicated next to the ray.  A select number of hyperbolic components are labeled with the period of the component.  External rays drawn by Wolf Jung's program Mandel.}
\label{external rays} 
\end{figure}

A hyperbolic component $H$ in $\cM_d$ is equipped with a $(d-1)$-to-1 map $\lambda_H: H \rightarrow \mathbb{D}$, given by the multiplier of the attracting cycle; this map extends continuously to the boundary; the point $0$ has a unique preimage under the multiplier map, known as the {\em center} of the hyperbolic component.  Given $\frac{p}{q} \in \mathbb{Q} / \mathbb{Z}$, the preimage under $\lambda_H$ of the ray $\{ re^{2 \pi i p/q} : 0 < r < 1 \} \subset \mathbb{D} \setminus \{ 0 \}$ is a collection of $d-1$ disjoint curves in the component (known as {\em internal rays}), which land at parabolic points on $\partial H$.  In this case, we say that the wake is a {\em $\frac{p}{q}$-subwake} of $H$.  Conversely, each parabolic point is the landing of some internal ray.   See Figure \ref{internal rays} for an illustration.  

\begin{figure} [t]
\includegraphics[width=2.05in]{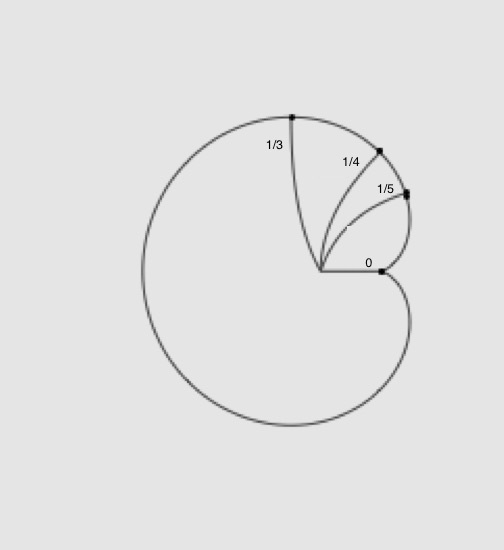}

\caption{ \small Select internal rays of the period $1$ component of the Mandelbrot set, with angles indicated.  Internal rays drawn by Wolfram's Mathematica.}
\label{internal rays}
\end{figure}

For example, if $H$ is the period $1$ component, some preimage of the point $e^{2 \pi i / k} \in \partial \mathbb{D}$ will land at the point $c_k \in \partial H$ of Proposition \ref{standard pair}.  

If $H$ is a hyperbolic component of $\cM_d$ of period $> 1$, there is a unique point $c_H$ on the boundary of $H$ so that both $\lambda(c_H) = 1$ and $\cM_d \setminus \{ c_H \}$ consists of two connected components; this is the {\em root} of $H$.  There will be exactly two rays landing at $c_H$, and we correspondingly define the {\em width} of the component $H$ to be the width of the wake at $c_H$.

Our key tool towards the main theorem is the so-called {\em wake formula}, which was folklore, eventually due to Bruin-Kaffl-Schleicher in \cite{BKL} for $d=2$  and Kauko for general $d$ (see \cite{Kauko}):

\begin{prop} \label{width formula} Let $H$ be a hyperbolic component of $\cM_d$ with period $k$ and width $|H|$.  Let $w_{p/q}$ be any $\frac{p}{q}$-subwake of $H$.  Then 
$$|w_{p/q}| = \frac{|H|}{d-1} \frac{(d^k-1)^2}{d^{qk}-1}.$$
\end{prop}

\begin{cor} \label{1/2 rays} Let $H$ be a hyperbolic component of the degree $d$ Mandelbrot set.  Then the $1/2$-subwakes of $H$ are precisely the set of subwakes of $H$ with maximal possible width.
\end{cor}

We provide now a key proposition towards Theorem \ref{main theorem}.

\begin{prop} \label{etarational} Under the hypothesis of Theorem \ref{escape relation}, there exists a component $X_0$ of the preimage of $\mathbb{C} \setminus \cM_d$ under $\psi_2$ such that the real number $\eta$ of Equation \ref{phirelation} is rational.  
\end{prop}

\begin{proof} We have two cases.  Suppose first that there exists $t \in X$ satisfying the following:

\begin{enumerate}
\item $t_0$ is not a branch point for $\psi_1$ or $\psi_2$, 
\item both $\psi_1(t)$ and $\psi_2(t)$ are PCF parameters, and
\item $\psi_1(t)$ or $\psi_2(t)$ is a Misiurewicz point. 
\end{enumerate}

Suppose without loss of generality that $\psi_1(t)$ is Misiurewicz; since $\psi_1(t) \in \partial \cM_d$, every open neighborhood of $\psi_1(t)$ contains parameters $c$ such that $|\Phi(c)| > 0;$ by Equation \ref{phirelation}, the same holds for $\psi_2(t)$.  Choose $X_0$ so that $t \in X_0$, and write $\mathcal{R}(\theta_1)$ and $\mathcal{R}(\theta_2)$ for the rational external rays landing at $\psi_1(t)$ and $\psi_2(t)$, respectively.  By Equation \ref{phirelation}, we have
$$d_2 \theta_1 - d_1 \theta_2 = \eta,$$
and so $\eta$ is rational as desired.

Suppose now that the conditions above are not satisfied for any $t \in X$.  Call $B$ the set of branch points of the projection maps $\psi_i$.  Since $B$ is a finite set, then by hypothesis, there exists some $t_0 \in X \setminus B$ such that both $\psi_1(t_0)$ and $\psi_2(t_0)$ are centers of hyperbolic components.  In fact, we may choose $t_0$ so that the component $H_2$ of $\cM_d$ which has center $\psi_2(t_0)$ is far from the branch points in the following sense: there exists an open neighborhood $U$ of $t_0$ such that $\psi_2(U)$ is simply connected, $H_2 \subset \psi_2(U)$, and there exists a parabolic parameter $c \in \partial H_2$ such that $c$ is the root of a component $H'_2$ satisfying $\overline{H'_2} \subset \psi_2(U)$.  Therefore we have a well-defined analytic function 
$$h(z) := \psi_1 \circ \psi_2^{-1} : U \rightarrow h(U).$$

Since $h$ is an open map, $h(H_2)$ is a component of $\mathbb{C} \setminus \partial \cM_d$ which contains the PCF parameter $\psi_1(t_0)$, so is hyperbolic.  Since $U$ contains no branch points, $h(c)$ lies on the boundary of two stable components, so by Theorem 4.1 of \cite{Schleicher1}, $h(c)$ is a parabolic parameter.  Choose $t \in X$ such that $\psi_1(t) = h(c)$ and $\psi_2(t) = c$, and choose $X_0$ so that $t \in X_0$.  By Equation \ref{phirelation}, any rational rays $\mathcal{R}(\theta)$ landing at $c$ and $\mathcal{R}(\theta')$ landing at $h(c)$ satisfy the relation
$$d_2 \theta' - d_1 \theta = \eta,$$
and we conclude that $\eta$ is rational as desired.

\end{proof}

We are now ready to prove the remaining significant result towards Theorem \ref{main theorem}.

\begin{theorem} \label{trivial relation} Assume the hypothesis of Theorem \ref{escape relation}.  Then there exists an open subset $U$ of the complex plane containing infinitely many PCF parameters on which an analytic branch of $\psi_1 \circ \psi_2^{-1}$ is given by $z \mapsto \zeta z$, for   some $(d-1)$st root of unity $\zeta$.
\end{theorem}

\begin{proof}
Fix an integer $m > 2$. We will define a neighborhood $U(m)$ as follows: define $c_m$ to be the landing point of the external ray $\mathcal{R}(\frac{d}{d^m-1})$; by Proposition \ref{standard pair}, this point lies on the main component.  By the preceding discussion, there also exists an internal ray $r_m$ landing at $c_m$, as well the internal ray $r_0$ of angle zero which lands at $c_0 = (d-1)/d^{d/(d-1)}$.  The union

$$C := r_0 \cup r_m \cup \mathcal{R}(0) \cup \mathcal{R}\left( \frac{d}{d^m-1} \right) \cup \{ 0 \}$$
is a curve such that  $\mathbb{C} \setminus C$ has two (simply connected) components.  Define $U(m)$ to be the component of $\mathbb{C} \setminus C$ which contains the hyperbolic component with root $c_m$; in other words, the component containing parameters of arbitrarily small argument (see Figure \ref{external rays}, \ref{internal rays}).  By Proposition \ref{standard pair}, $U(m)$ contains infinitely many PCF parameters.

Note now that $m$ may be chosen sufficiently large so that $U(m)$ omits the images of the branch points of $\psi_1$ and $\psi_2$. Therefore we may define an analytic branch $h(z) = \psi_1 \circ \psi_2^{-1}$ on $U(m)$.  By Equation \ref{phirelation}, $h$ sends external rays to external rays, and there exists $0 < \ell \leq d_2$ (given by choice of branch) such that $h$ acts on external angles $0 < \theta < \frac{d}{d^m-1}$ by 
$$\theta \mapsto \frac{d_1}{d_2} \theta + \eta + \frac{\ell}{d_2}.$$

Denote such a choice of $U(m)$ simply by $U$, and write $\eta + \frac{\ell}{d_2} =  \frac{a}{b}$ in lowest terms (this is possible by Proposition \ref{etarational}).  Note then that if $\mathcal{R}(\theta_1)$ and $\mathcal{R}(\theta_2)$ land together in $U$, their images under the continuous map $h$ must land together, and the wake of the image rays has width $\frac{d_1}{d_2} |\theta_2 - \theta_1|.$   

\begin{prop} 
\label{prop same period}
Let $U$ and $h$ be defined as above.  Then every hyperbolic component of $U$ is sent by $h$ to a hyperbolic component of the same period.
\end{prop}

\begin{proof}[Proof of Proposition~\ref{prop same period}.]
Suppose $H$ is a hyperbolic component of period $N$ contained in the neighborhood $U$ with root $c$ and wake $W$.   Since $h$ is an open map on $U$ which preserves rational external rays, $h(H)$ is also a hyperbolic component, say of period $N'$, root $h(c)$, and wake $h(W)$.  Choose any landing point of an internal ray of $H$ with angle $1/2$; call this point $c_{1/2}$, and its subwake $H_{1/2}$, noting that $c_{1/2}$ is the root of a hyperbolic component of period $2N$.  Since $h$ has a linear action on external angles, Corollary \ref{1/2 rays} guarantees that $h(H_{1/2})$ is one of the $1/2$-subwakes of $h(H)$; call it $h(H)_{1/2}$.  The width formula of Proposition \ref{width formula} computes:
$$\frac{|h(H)|}{d-1} \frac{(d^{N'}-1)^2}{d^{2N'}-1} = |h(H)_{1/2}| = |h(H_{1/2})| = \frac{d_1}{d_2} |H_{1/2}| = \frac{d_1}{d_2} \frac{|H|}{d-1} \frac{(d^N-1)^2}{d^{2N}-1}.$$
Since $|h(H)| = \frac{d_1 \ell}{d_2} |H|,$ the right- and left-hand sides of the equation above imply that $N = N'$, as desired.
\end{proof}

For any hyperbolic component of period $N$, the rays landing at the root of the component have denominators which divide $d^N-1$ (this is again standard; see Chapter 8 of \cite{Douady-Hubbard}), and therefore the width of a period $N$ component has denominator which is a divisor of $d^N-1$.  Since $h$ fixes the period of any hyperbolic component in $U$, and $U$ contains components of period $N$ and width $\frac{d-1}{d^N-1}$ for all $N > m$ (see Proposition~\ref{standard pair}), we deduce that $\frac{d_1}{d_2} \cdot \frac{d-1}{d^N-1}$ has denominator which is a divisor of $d^N-1$ for all $N > m$.  We conclude that $d_2$ divides $d_1 (d-1)$, and so the map that $h$ induces on $\mathbb{R} / \mathbb{Z} \cap (0, \frac{d}{d^m-1})$ is simply
$$\theta \mapsto \frac{k}{d-1} \theta + \frac{a}{b}$$
for some $1 \leq k \leq d-1$.

By the same arguments and choosing $N$ sufficiently large, the ray 
$$h\left(\mathcal{R}\left(\frac{1}{d^N-1}\right)\right) =  \mathcal{R}\left(\frac{d-1}{d^N-1}\cdot \frac{k}{d-1} + \frac{a}{b}\right)$$
lands at a hyperbolic component of period $N$, and so $\frac{k}{d^N-1} + \frac{a}{b}$ has denominator dividing $d^N-1$.  Since $a/b$ is in lowest terms, we conclude that $b$ divides $d^N-1$ for all $N$ sufficiently large.  Choosing $M$ and $N$ large and coprime, the greatest common divisor of $d^M-1$ and $d^N-1$ is $d-1$, so we have $b \mid (d-1)$.  

We now have integers $k$, $j$ so that $h$ acts on external angles of $U$ by 
$$\theta \mapsto \frac{k}{d-1} \theta + \frac{j}{d-1}.$$
However, we know that the ray of angle $\frac{1}{d^N-1}$ maps to a ray with denominator dividing $d^N-1$ for all $N$ sufficiently large; in other words, 
$$k + j(d^N-1) \ \equiv  0  \ \mod (d-1)$$
for all $N$ sufficiently large.  Thus $k = d-1$, and so $h$ acts on external angles by translation by $\frac{j}{d-1}$; that is, $h$ acts on external rays as multiplication by some $(d-1)$st root of unity $\zeta$.  

By definition of $\Phi$,
$$\Phi(\zeta z) = \zeta \Phi(z)$$
for all $z \in \mathbb{C} \setminus \cM_d$, so $h$ coincides with the map $z \mapsto \zeta z$ on $(\mathbb{C} \setminus \cM_d) \cap U$, and thus on the entire domain $U$.
\end{proof}

The proof of the main theorem follows easily from Theorem~\ref{trivial relation}.

\begin{proof}[Proof of Theorem \ref{main theorem}.] First we note that indeed, if $C$ has the form (1),~(2)~or~(3) as in the conclusion of Theorem~\ref{main theorem}, then it contains infinitely many points $(a,b)$ with both coordinates PCF parameters, i.e., both $z^d+a$ and $z^d+b$ are PCF polynomials. For curves of the form (1)~or~(2), this fact is obvious, while for curves of the form (3), we note that once $f_c(z):=z^d+c$ is PCF, then also $f_{\zeta c}(z):=z^d+\zeta c$ is PCF (where $\zeta^{d-1}=1$)  because $\zeta^{-1}f_{\zeta c}(\zeta z) = f_c(z)$. Also, there exist infinitely many $c\in\Qbar$ such that $f_c$ is PCF.

So, from now on, assume $C$ be an irreducible plane curve containing infinitely many $(a, b)$ such that $z^d+a$ and $z^d+b$ are both PCF. Since the PCF parameters are algebraic numbers, we conclude that $C$ is defined over $\Qbar$. If $C$ does not project dominantly onto one of the two coordinates of $\A^2$ then, without loss of generality, we may assume $C=\{t_0\}\times \A^1$ for some $t_0\in\Qbar$. But then by the hypothesis satisfied by $C$, we conclude that $t_0$ is a PCF parameter, i.e., $C$ has the form (1) as in the conclusion of Theorem~\ref{main theorem}.

So, from now on, we assume $C$ projects dominantly onto both coordinates of $\A^2$.  Let $\pi : X \rightarrow C$ be a nonsingular projective model of $C$; therefore $X$ is defined over some number field $K$.  Write $\pi_1$ and $\pi_2$ for the projection maps of $C$ onto the axis of $\A^2$, and let $\psi_i = \pi_i \circ \pi$ for $i=1,2$.  Then we can apply Theorem~\ref{escape relation} and deduce Theorem~\ref{trivial relation}. Thus there exists a $(d-1)$st root of unity $\zeta$ such that for infinitely many $c \in \cM_d$, $\zeta c = \psi_1(t_c)$ and $c = \psi_2(t_c)$ for some $t_c \in X(\C)$; accordingly, there exist infinitely many $c\in\C$ such that $(\zeta c, c) \in C(\C)$.  Since $C$ is irreducible, we conclude the proof of Theorem~\ref{main theorem}.
\end{proof}

\def\cprime{$'$}

\end{document}